\documentclass[]{article}
\usepackage{amssymb,amsmath,amsfonts, amstext, graphicx}
\usepackage{dsfont}
\pdfoutput=1 

\usepackage[]{hyperref}
\hypersetup{pdftitle={Note on the Existence of Minimizers for Variational Geometric Active Contours},
colorlinks=true,citecolor=blue,linkcolor=red}

\DeclareMathOperator*{\argmin}{arg\,min}

\newtheorem{thm}{Theorem}[section]
\newtheorem{prop}{Proposition}[section]
\newtheorem{proof}{Proof}[section]
\newtheorem{definition}{Definition}[section]
\newtheorem{rmk}{Remark}[section]

\begin{document}

\title{Note on the Existence of Minimizers for Variational Geometric Active Contours}

\author{%
El Hadji S. Diop\thanks{Department of Mathematics, University of Thies, Thies BP 967, Senegal. Email: ehsdiop@hotmail.com}
\and 
Val\'erie Burdin\thanks{Image and Information Department, Institut Mines Telecom - Telecom Bretagne, Technopole Brest-Iroise CS 83818, 29238 Brest Cedex 3, France. Email: valerie.burdin@imt-atlantique.fr}
\and 
V. B. Surya Prasath\thanks{Division of Biomedical Informatics, Cincinnati Children's Hospital Medical Center, Cincinnati, OH 45229 USA. Also with the Departments of Biomedical Informatics, Pediatrics, Electrical Engineering and Computer Science, University of Cincinnati, OH, USA. Email: surya.iit@gmail.com, prasatsa@uc.edu}}
\date{}
\maketitle
\begin{abstract}
We propose here a proof of existence of a minimizer of a segmentation functional based on a priori information on target shapes, and formulated with level sets.  The existence of a minimizer is very important, because it guarantees the convergence of any numerical methods (either gradient descents techniques and variants, or PDE resolutions) used to solve the segmentation model. This work can also be used in many other segmentation models to prove the existence of a minimizer.

\end{abstract}
\section{Introduction}\label{sec:intro}

Image segmentation is still subjects of intensive researches due to emerging acquisition techniques and capacity of image processing and computer vision models in modeling and solving real life problems. In many applications, e.g. medical imaging, it is essential by just a first step towards the true application (e.g. results analysis, therapeutic evaluation, $\dots$). Accuracy is then very important, and due to that fact, we developed a segmentation model \cite{Diop2013} for $X$-$rays$ medical images that suffer from a poor contrast between objects of interests (femur or tibia for instance) and other structures (soft tissues, pelvis, hip, fibula, tarsus, metatarsus, $\dots$), problems of edge salience, occlusions phenomena in bones joints, $\dots$. Prior shapes were incorporated in a level set-based variational formulation of the segmentation problem, and then, as common, associated Euler-Lagrange equations were derived constituting our segmentation model. A proof of existence of a minimizer of a segmentation energy is not always provided, the segmentation model is solved numerically (gradient descent, partial differential equation numerical resolutions, $\dots$), rather. The main reason is the complexity of the designed segmentation functionals, and then, many criteria (e.g. convexity) are not usually satisfied in order to apply classical optimization results. Many techniques have been developed e.g. convex relaxation techniques \cite{Pock2009,Brown2012}, which are weak formulation of the original problem. Although the existence of a minimizer was not proven, the efficiency and robustness of the segmentation model \cite{Diop2013} were showed on various image types; namely, on synthetic images, digitally reconstructed images, and real radiographic images. Also, quantitative evaluations of obtained segmentation results were provided. Here, we propose a proof of existence of a minimizer for the variational segmentation functional \cite{Diop2013}. Different works \cite{Chen2002,Bresson2006,Chung2009,Li2010} and more recently \cite{Moren2014,Wu2015,prasath2015cell,Ozere2015,pelapur20173d} pursued the same goal, even if the proposed segmentation functionals were different.

The article is organized as follows. The segmentation functional is presented in Section~\ref{sect2}. The proof of the existence of a minimizer is proposed in Section~\ref{sect3}. We conclude with some perspectives in Section~\ref{sect4}.

\section{Variational Image Segmentation Model}\label{sect2}

Let $I:\Omega\rightarrow \mathds{R}, I =\lbrace I(x), x\in \Omega\rbrace$ be a continuous image, $\Omega\subset\mathds{R}^2$ its domain assumed to be open, bounded and with a Lipschitz boundary denoted by $\partial\Omega$.

\subsection{Prior Shapes Design}

We consider a set $\lbrace C_n\rbrace_{1\leq n \leq N}$ of $N$ aligned contours embedded as signed distance functions denoted
by $\phi_n$, $\forall~n\geq 0$. Let $\lbrace\phi_n\rbrace_{1\leq n \leq N}$ be the training dataset, and denote
$\bar{\phi} = \displaystyle\frac{1}{N}\sum_{n=1}^N \phi_n$ its mean. In order to learn the local deformations of the dataset,
we follow \cite{Leventon2000} and perform a $PCA$ on $\lbrace\phi_n\rbrace_{1\leq n \leq N}$. Advantages in doing that on $(\phi_n)_n$
instead of on $(C_n)_n$ were already discussed \cite{Diop2013}. So, this means looking for the best orthonormal basis $\lbrace
u_k\rbrace_{1\leq k \leq K}$ s.t. the projection of $\phi_n$ on $(u_k)_k$ has a minimal distance in $L^2$ sense. Let $U$ be composed
with $p$ first columns of $V$ holding the orthogonal modes of variations, new shapes can be built then by: $\phi = \bar{\phi} 
+ U\lambda$, where $U = [u_i]_{1\leq i \leq p}$ is the eigenvectors matrix and $\lambda = [\lambda_i]^{\text{t}}_{1\leq i \leq 
p}$ the shape parameters vector.

\subsection{Segmentation Functional and Variational Formulation}

The segmentation functional is defined as follows~\cite{Diop2013}:
\begin{equation} 
\label{F}
F(\varphi,\lambda,V,I_{in},I_{out}) = \dfrac{\alpha}{2} F_1(\varphi) + F_2(\varphi,\lambda,V) + \beta F_3(\varphi) + \nu F_4(\lambda,V,I_{in},I_{out}), 
\end{equation}
where $\alpha, \beta, \nu >0$ for counterbalancing different energy terms defined as follows: 
\begin{align}
&F_1(\varphi) = \int_\Omega \left( \lvert\nabla\varphi\rvert - 1\right)^2\mathrm{d} x, \label{F1}
\\
&F_2(\varphi,\lambda,V) = \int_\Omega [\xi g + \dfrac{\gamma}{2}\phi^2(\lambda,h(x))]\delta(\varphi)
\lvert\nabla\varphi\rvert\mathrm{d} x, \label{F2}\\
&F_3(\varphi) = \int_\Omega g\,H(-\varphi)\mathrm{d} x, \label{F3} \text{ and }\\ 
& F_4(\lambda,V,I_{in},I_{out}) = \int_\Omega \lvert I - I_{in}\rvert^2 + \mu\lvert\nabla I_{in}\rvert^2 H(\phi(\lambda, h(x)))\, \mathrm{d} x \nonumber \\
 & + \int_\Omega \lvert I - I_{out}\rvert^2 + \mu\lvert\nabla I_{out}\rvert^2 (1 - H(\phi(\lambda,h(x))))\, \mathrm{d} x  +
 \zeta\int_\Omega d\mathcal{H}^1(C(\lambda,h(x))); \label{F4}
\end{align}
with $\varphi$ being the level set function; $\xi$ and $\gamma$ are positive parameters; $\delta(\cdotp)$
is the unidimensional Dirac distribution; $g$ is a positive and strictly non increasing function, typically: $g=\dfrac{1}{1+\eta\lvert\nabla G_\sigma\star I\rvert^2}$, $\star$ standing for the convolution operator, $\eta >0$, $G_\sigma$ is a Gaussian kernel with a variance $\sigma$. $I_{in}$ and $I_{out}$ designate the Mumford-Shah functional terms;
$\nu$ is a positive parameter. $H(\cdotp)$ is the Heaviside function,~i.e.~$H:\mathds{R}\longrightarrow\lbrace 0, 1\rbrace$, s.t. $H=1$ in $\mathds{R}^{+}$, $H=0$ in $\mathds{R}_{\star}^{-}$, and $\delta = H'$ in the sense of distributions. $\phi$ is the prior shape; $\lambda$ is the shape parameters vector; $h:\Omega\longrightarrow \Omega$; $x\longmapsto h(x) = \tau R_\theta x + T$ accounts the spatial rigid transformations achieved through a vector $V$ holding the scale parameter factor $\tau$, the rotation matrix $R_\theta$ of angle $\theta$ and the translation vector $T=[T_x,\,T_y]$; $C(\lambda,h(x)) = \lbrace x\in \Omega$ s.t. $\phi(\lambda,h(x)) = 0 \rbrace$; $\mathcal{H}^1$ is the 
one dimensional Hausdorff measure.

\begin{itemize}
	\item[$\bullet$] $F_1$ (\ref{F1}) guarantees the signed distance function property, and avoids the re-initialization process. In fact, when keeping $\lvert\nabla\varphi\rvert$ bounded, one ensures the correct computations of $\varphi$ derivatives. The most common way is to apply the re-initialization procedure by periodically solving the PDE \cite{Sussman1994,Zhao1996,Chan2001}:
\begin{equation}
\label{re_init}
 \left\{
\begin{array}{c}
 \dfrac{\partial\varphi}{\partial t} = \text{sign}(\varphi_0)(1-\lvert\nabla\varphi\rvert) \\
 \varphi(x,0) = \varphi_0(x),
\end{array}
 \right. 
\end{equation}
where $\varphi_0$ refers to the level set to be re-initialized. 
	\item[$\bullet$] $F_2$ (\ref{F2}) makes the active contour evolve towards high gradients areas, and also towards similar regions of the prior shape. Minimizing $F_2$ increases the similarity between the active contour and the shape prior. 
	\item[$\bullet$] $F_3$ (\ref{F3}) could be interpreted as a weighted area term of the target object of interests. Notice that 
if $g$ equal to $1$, then $F_3$ is the area of the region of the object of interests; i.e.~$\lbrace x\in \Omega, \text{ s.t. } \varphi(x) < 0 \rbrace$. Minimizing it provides another force pushing the active contour quickly towards the edges.
	\item[$\bullet$]  $F_4$ (\ref{F4})  is a modified version of the complete Mumford-Shah functional energy \cite{Mumford1989}, which was the most suitable model for the bi-planar $X$-ray images we dealt with, since anatomic regions of interests  have almost homogeneous gray level image intensities.
\end{itemize}

The level set formulation of the variational problem is thus given by:
\begin{align} 
\label{form_var}
\varphi,\lambda,V,I_{in},I_{out} & = \argmin_{\varphi,\lambda,V,I_{in},I_{out}} F(\varphi,\lambda,V,I_{in},I_{out}) \nonumber\\
                                 & = \argmin_{\varphi,\lambda,V,I_{in},I_{out}}\dfrac{\alpha}{2} F_1(\varphi) + F_2(\varphi,\lambda,V) \nonumber\\ 
                                 &+  \beta F_3(\varphi) + \nu F_4(\lambda,V,I_{in},I_{out}).
\end{align}
By the means of calculus of variations \cite{Evans1998,Aubert2000}, we solve the Euler-Lagrange equations associated to the variational formulation (\ref{form_var}) to finally obtain the evolution equations that constituted the segmentation model~\cite{Diop2013}.

\section{Minimum Existence of the Segmentation Functional}\label{sect3}

Let us recall first definitions and most important properties related to the space of Bounded Variation ($BV$)~\cite{Evans1992,Giusti1994,Ambrosio2000,Ambrosio2001}:
\begin{definition}
Let $u\in L^1(\Omega)$, and denote $$\int_\Omega |Du| = \sup\left\lbrace\int_\Omega u~\text{div}(\varphi)\mathrm{d} x;~\varphi=(\varphi_1,\varphi_2)
\in C_0^1(\Omega,\mathds{R}^2) \text{ and } \lVert\varphi\rVert\leq 1\right\rbrace,$$
where $C_0^1(\Omega)$ is the space of continuously differential functions with compact support in $\Omega$, $\text{div}(\varphi)
=\dfrac{\partial\varphi_1}{\partial x_1}+\dfrac{\partial\varphi_2}{\partial x_2}$ where derivatives are taken in a distributional
sense, and $\lVert\varphi\rVert=\lVert(\varphi_1^2+\varphi_2^2)^{1/2}\rVert_{L^\infty(\Omega)}$. The space of functions of bounded variation on $\Omega$, denoted BV($\Omega$), is the set of functions defined by $$BV(\Omega) = \left\lbrace u\in L^1(\Omega);
s.t. \int_\Omega |Du|<\infty\right\rbrace.$$
\end{definition}
\begin{rmk}
BV($\Omega$) is a Banach space when endowed with the norm $\lVert u\rVert_{\text{BV}(\Omega)}=\lVert u\rVert_{L^1(\Omega)}
+ \int_\Omega |Du|$.
\end{rmk}
\begin{prop}[Lower semi continuity]
 Let $(u_n)_n$ be a sequence of functions $\in\text{BV}(\Omega)$ s.t. $u_n\underset{n}\longrightarrow u$ in $L^1_{loc}(\Omega)$. Then,
 one has: $$\int_\Omega|Du|\leq  \underset{n}{\lim\inf}\int_\Omega |D u_n|.$$
\end{prop}
\begin{prop}[Approximations]
 Let $u\in\text{BV}(\Omega)$. Then, there exists a sequence $(u_n)_n$ of functions belonging to $\text{BV}(\Omega)\cap C^\infty(\Omega)$
s.t. : $$u_n\underset{n}\longrightarrow u \text{ in } L^1(\Omega) \text{ and } \int_\Omega|D u_n|\underset{n}\longrightarrow \int_\Omega 
|D u|.$$
\end{prop}
\begin{prop}[Compactness]
\label{compac}
Let $(u_n)_n$ be a uniformly bounded sequence of functions $\in\text{BV}(\Omega)$. Then, there exists a subsequence $(u_{\rho(n)})_n$ 
of $(u_n)_n$ and a function $u\in\text{BV}(\Omega)$ s.t.: $u_n\underset{n}\longrightarrow u \text{ in } L^1_{\text{loc}}(\Omega)$.
\end{prop}
Thanks to preceding results on BV spaces, we state and prove the main result of the paper:
\begin{thm}
 There exists $(\varphi,\lambda,V,I_{in},I_{out})$ as a minimizer of the segmentation functional $F$ (\ref{F}).
\end{thm}
\begin{proof}
The proof is based on the direct method of the calculus of variations \cite{Giusti2000}, compactness theorems and  standard techniques of the calculus of variations on BV spaces and its variants.\\
Let $\Omega=[0;~255]^2$, $\lambda=(\lambda_i)_{i=1}^p\in \Omega_\lambda=\underset{1\leq i\leq p}\prod[-3x_i;3x_i]$; $x_i,i\in[[1;~p]]$, being the resulted eigenvalues from the 
$PCA$, $V=[\tau,~\theta,~T]\in \Omega_V=[0;~255]\times[-\pi;~\pi]\times[0;~255]^2$, and let $J=I_{in}$ and $\bar{J}=I_{out}$.\\

Let now $(\varphi_n)_n, (\lambda_n)_n, (V_n)_n, (J_n)_n$ and $(\bar{J}_n)_n$ be minimizing sequences of $F$ (\ref{F});~i.e.
\begin{equation}
\underset{n}\lim~ F(\varphi_n, \lambda_n, V_n, J_n, \bar{J}_n)=\inf F.
\end{equation}

\noindent{\bf Step 1:} Let us consider $G$ defined as 
\begin{equation}
G(\varphi)=\dfrac{\alpha}{2}\int_\Omega(|\nabla\varphi|-1)^2 + \beta\int_\Omega g H(-\varphi) \text{ and }\Psi(\varphi)=(|\nabla\varphi|-1)^2.
\end{equation}
Then, $\forall~n\in\mathds{N}$, one has: 
\begin{equation}
G(\varphi_n)=\dfrac{\alpha}{2}\int_\Omega \Psi(\varphi_n)+g H(-\varphi_n).
\end{equation}

$\bullet$ $(\varphi_n)_n$ is a family of level set functions; one can assume then $\varphi_n$ belong to $C^1(\Omega)$, $\forall~n$. For all $n$, let $\Psi_n = \Psi(\varphi_n)$. In addition, since $\Omega$ is compact, then, $\varphi_n$ are bounded on $\Omega$, $\forall~n\geq 0$; thus, $\Psi_n\in\text{BV}(\Omega)$. The compactness result (Proposition~\ref{compac}) guarantees the existence of a subsequence $(\Psi_{\rho(n)})_n$ extracted from $(\Psi_n)$~s.t.: 
\begin{equation}
\Psi_{\rho(n)}\underset{n}\longrightarrow\Psi \text{ in } L^1(\Omega) \text{ sense, and } \int_\Omega |D\Psi|\leq  \underset{n}{\lim \inf}\int_\Omega |D\Psi_{\rho(n)}|.
\end{equation}
 Since $(\Psi_n)_n$ is positive, then, $\Psi\geq 0$; therefore, one has:
\begin{equation}
\label{G1}
\int_\Omega\Psi_n\underset{n}\longrightarrow\int_\Omega\Psi.
\end{equation}

$\bullet$ Let $D_n=\lbrace x\in\Omega;~\varphi_n\leq 0\rbrace$, $n\in\mathds{N}$. Then, $\forall~n\geq 0$, $H(-\varphi_n)=\chi_{D_n}$.\\
Since $g$ is bounded ($-1\leq g \leq 1$) in $\Omega$ and $(\chi_{D_n})_n\subset\text{BV}(\Omega)$, there exists then a function
$k\in\text{BV}(\Omega)$ and a subsequence $(\chi_{D_{\rho(n)}})_n$ (without loss of generality, we take the same extractor function
$\rho$ since the composition of extractors is also an extractor) s.t. $g\chi_{D_{\rho(n)}}\underset{n}\longrightarrow g\,k$ in $L^1(\Omega)$ sense. As in the previous case, the positivity of $(g\cdotp\chi_{D_{\rho(n)}})_n$ and $k$ yields then:
\begin{equation}
\label{G2}
\int_\Omega g\chi_{D_{\rho(n)}}\underset{n}\longrightarrow~\int_\Omega g\,k.
\end{equation}
(\ref{G1}) and (\ref{G2}) prove the existence of a minimizer for $G$.\\

\noindent{\bf Step 2:}  Let us consider $K$ in the following 
\begin{multline}
K(\lambda,V,J,\bar{J}) = \int_\Omega \lvert I - J\rvert^2 + \mu\lvert\nabla J\rvert^2 H(\phi(\lambda, h(\cdotp)))
+ \int_\Omega \lvert I - \bar{J}\rvert^2 \\
+ \mu\lvert\nabla \bar{J}\rvert^2 (1 - H(\phi(\lambda,h(\cdotp)))).
\end{multline}
 Previously defined minimizing sequences also satisfy then: \\$\underset{n}\lim~K(\varphi_n, \lambda_n, V_n, J_n, \bar{J}_n)=\inf K$. $H$ can be written as an indicator function by considering the set $\Theta=\lbrace x\in\Omega;~\phi(\lambda,h(x))\geq 0\rbrace$; thus, one has: \\$H(\phi(\lambda,h(x)))=\chi_\Theta$. It is clear that $\chi_\Theta$ depends continuously on $\lambda$ and $V$ 
parameters; since that is the case for $\Theta$. Let us introduce the function 
\begin{equation}
\varsigma(\cdotp)=|I - \cdotp|^2+|\nabla \cdotp|^2.
\end{equation}
 Then, one has: 
\begin{equation}
K(\lambda,V,J,\bar{J}) = \int_\Omega \varsigma(J)\chi_\Theta + \varsigma(\bar{J})(1-\chi_\Theta).
\end{equation}
 Let then introduce sequences $(\chi_{\Theta_n})_n$ and $(\varrho_n)_n$, defined $\forall~n\in\mathds{N}$, by:
\begin{equation} 
 \chi_{\Theta_n} = \chi_{\Theta}(\lambda_n,V_n) \text{ and } \varrho_n=\varsigma(J_n)\chi_{\Theta_n} + \varsigma(\bar{J}_n)(1-\chi_{\Theta_n}).
\end{equation} 
Then, thanks to the compactness-lower semi continuity result in $GSBV(\Omega)$ \cite{Ambrosio2013}, there exists $\varrho\in~GSBV(\Omega)$ and a subsequence $\varrho_{\rho(n)}$ s.t. $\varrho_{\rho(n)}
\underset{n}\longrightarrow \varrho$ in $L^p$ sense; $1\leq p <2$; moreover, one has: 
\begin{equation}
\int_\Omega\varsigma(J)\chi_\Theta\leq\underset{n}{\lim \inf}\int_\Omega \varsigma(J_n)\chi_{\Theta_{\rho(n)}} 
\end{equation}
and
\begin{equation}
\int_\Omega\varsigma(\bar{J})(1-\chi_\Theta)\leq\underset{n}{\lim \inf}\int_\Omega \varsigma(\bar{J}_n)(1-\chi_{\Theta_{\rho(n)}}).
\end{equation}
Thus, the existence of a minimizer of $K$.\\

\noindent{\bf Step 3:} Let \\$L(\varphi,\lambda,V) = \int_\Omega [\xi g + \dfrac{\gamma}{2}\phi^2(\lambda,h(x))]\delta(\varphi)
\lvert\nabla\varphi\rvert\, \mathrm{d} x + \zeta\int_\Omega d\mathcal{H}^1(C(\lambda,h(x)))$.\\

$\bullet$ One has \cite{Evans1992}: 
\begin{equation}
\int_\Omega d\mathcal{H}^1(C(\lambda,h(x))) = \int_\Omega\delta(\phi(\lambda, h(x))) |\nabla
\phi(\lambda, h(x))|\mathrm{d} x.
\end{equation}
$\Omega$ is a compact set, and since $(\phi_i)_i$ is a family of signed distance functions, then,
$\phi$ is smooth enough. Let us assume $\phi\in C^1(\Omega)$ and consider again the set $\Theta=\lbrace x\in\Omega;~\phi(\lambda,h(x))\geq 0\rbrace$ and the subsequence $(\chi_{\Theta_n})_n$ defined $\forall~n\in\mathds{N}$ by $\chi_{\Theta_n} = \chi_{\Theta}(\lambda_n,
V_n)$. Then, one also has \cite{Evans1992,Ambrosio2001}:
\begin{equation}
\int_\Omega d\mathcal{H}^1(C(\lambda,h(x))) = \int_\Omega|D\chi_\Theta|.
\end{equation} 
Because $\phi\in\text{BV}(\Omega)$, then, $\forall~n,~\Theta_n$ has finite perimeter;~i.e.~
$\chi_{\Theta_n}\in\text{BV}(\Omega),~\forall~n$. Consequently, there exists $k\in\text{BV}(\Omega)$ and a subsequence $\chi_{\Theta_{\rho(n)}}$ s.t. $\chi_{\Theta_{\rho(n)}}\underset{n}\longrightarrow k$ in $L^1(\Omega)$ sense. Since $\chi_{\Theta_{\rho(n)}}$ are characteristic functions, $k$ is also a characteristic function; let then $k=\chi_\Theta$.\\
Let now $\varphi\in C_0^1(\Omega,\mathds{R}^2)$~s.t.~$\lVert\varphi\rVert\leq 1$. Because $\chi_{\Theta_{\rho(n)}}\underset{n}\longrightarrow \chi_\Theta$ in $L^1(\Omega)$ sense, then, one has:
\begin{equation}
\int_\Omega \chi_\Theta\text{div}(\varphi)=\underset{n}\lim\int_\Omega\chi_{\Theta_{\rho(n)}}\text{div}(\varphi)
\leq\underset{n}\lim\int_\Omega|D\chi_{\Theta_{\rho(n)}}|.
\end{equation}
Thus, by taking the supremum over all $\varphi$, one has:
\begin{equation}
 \label{L1}
 \int_\Omega d\mathcal{H}^1(C(\lambda,h(x))) =\int_\Omega |D\chi_\Theta| \leq  \underset{n}{\lim \inf}\int_\Omega |D\chi_{\Theta_{\rho(n)}}|.
\end{equation}

$\bullet$ Let $W=\lbrace x\in\Omega;~\varphi(x)\geq 0 \rbrace$ and $f(\cdotp,\lambda,V)=\xi g(\cdotp) + \dfrac{\gamma}{2}\phi^2(\lambda,h(\cdotp))$
defined on $\Omega$. Following \cite{Evans1992,Chan2001}, then one has: 
\begin{equation}
F_2(\chi_W,\lambda,V)=\int f(\cdotp,\lambda,V)|D\chi_W|.
\end{equation}
Let us consider now the sequence $(f_n)_n$ defined for all $n\in\mathds{N}$ by $f_n(\cdotp)=f(\cdotp,\lambda_n,V_n)$ in $\Omega$, where 
$(\lambda_n)_n$ and $(V_n)_n$ are previously defined sequences. It is clear that $(f_n)_n$ is a positive sequence, since $f\geq 0$.
In addition, $g, \phi^2\in C^\infty(\Omega)$ and $\Omega$ is compact; then, $(f_n)_n\subset\text{BV}(\Omega)$. Therefore, there
exists subsequences $(\lambda_{\rho(n)})_n$ and $(V_{\rho(n)})_n$ s.t.: $f_{\rho(n)}(\cdotp)=f(\cdotp,\lambda_{\rho(n)},V_{\rho(n)})\underset{n}\longrightarrow~f(\cdotp)$ in $L^1(\Omega)$ sense, and 
\begin{equation}
\int_\Omega f(\cdotp,\lambda,V) \leq\underset{n}{\lim \inf}\int_\Omega f(\cdotp,\lambda_{\rho(n)},V_{\rho(n)}).
\end{equation}
 Tychonoff's theorem ensures that $\Omega, \Omega_V$ and $\Omega_\lambda$ are compact sets; then, there exists subsequences $(\lambda_{\rho(n)})_n$ and $(V_{\rho(n)})_n$ s.t:
$\lambda_{\rho(n)}\underset{n}\longrightarrow\lambda$ and $V_{\rho(n)}\underset{n}\longrightarrow V$. Since $f$ is continuous, then, one has $f(\cdotp,\lambda_{\rho(n)},V_{\rho(n)})\underset{n}\longrightarrow f(\cdotp,\lambda, V)$.\\

Following \cite{Chen2002}, let $\varepsilon> 0$. There exists then $N\in\mathds{N}$~s.t.~$\forall~n\geq N$, one has $|f_{\rho(n)}(\cdotp) - f(\cdotp)|\leq\varepsilon$. Let us define $\Psi=\lbrace\psi\in C_0^1(\Omega,\mathds{R});~\lVert\psi\rVert\leq 1 \text{ in } \Omega\rbrace$ and $\Psi_\omega=\lbrace\psi\in C_0^1(\Omega,\mathds{R});~\lVert\psi\rVert\leq \omega \text{ in } \Omega\rbrace$, for any positive and continuous function $\omega:\Omega\longrightarrow\mathds{R}$. 

 Now, let $n\in\mathds{N}$~s.t.~$n\geq N$ and let $\varphi\in\Psi_{f-\varepsilon}$. Then, $\varphi\in\Psi_f$.
We previously show that $\forall~n\in\mathds{N},~\varphi_n\in\text{BV}(\Omega)$ implying $W_n=\lbrace x\in\Omega;~\varphi_n(x)\geq 0\rbrace$ has a finite perimeter;~i.e.~$\chi_{W_n}\in\text{BV}(\Omega),~\forall~n\in\mathds{N}$. As in the previous case, there exists then a subsequence $\chi_{W_{\rho(n)}}$~s.t.~$\chi_{W_{\rho(n)}}\underset{n}\longrightarrow \chi_W$ in $L^1(\Omega)$ sense. Then, one has:
\begin{equation} 
 \int_\Omega \chi_W\text{div}(\varphi)=\underset{n}\lim\int_\Omega\chi_{W_{\rho(n)}}\text{div}(\varphi).
\end{equation} 
  Since $\Psi_{f-\varepsilon}  \subseteq\Psi_{f}$, then: 
\begin{equation}  
  \int_\Omega \chi_W\text{div}(\varphi)\leq \underset{n}\lim\int_\Omega f_{\rho(n)}|D\chi_{W_{\rho(n)}}|.
\end{equation}  
   Due to the fact 
that $n\geq N$, $\varphi\in\Psi_{f-\varepsilon}$ and $\varepsilon>0$ are arbitrary chosen; then: 
\begin{equation}
\underset{\varphi\in\Psi_{f-\varepsilon}}\sup \int_\Omega \chi_W\text{div}(\varphi)\leq \underset{n}\lim\int_\Omega f_{\rho(n)}|D\chi_{W_{\rho(n)}}|.
\end{equation}
 Thus,
by letting $\varepsilon\rightarrow 0$, one has:
\begin{equation}
\label{L2}
 \int_\Omega f|D\chi_W|=\underset{\varphi\in\Psi_f}\sup\int_\Omega \chi_W\text{div}(\varphi)\leq \underset{n}\lim\int_\Omega 
f_{\rho(n)}|D\chi_{W_{\rho(n)}}|.
\end{equation}
(\ref{L1}) and (\ref{L2}) prove the existence of a minimizer of $L$. 
\end{proof}
\section{Conclusion}\label{sect4}

A proof of existence of a minimizer of a non-convex variational segmentation functional is provide here. That is important, since it guarantees the good evolution of the segmentation model towards the desired objects, and it can be seen as a completion of a previous work where both the efficiency and robustness of the segmentation model were already demonstrated without that proof. However, since the functional is not convex, the existence of a minimum does not guarantee the uniqueness. We plan to investigate that for a future work.
\bibliographystyle{plain}
\bibliography{biblio}

\begin{thebibliography}{10}

\bibitem{Ambrosio2001}
Luigi Ambrosio, Vicent Caselles, Simon Masnou, and Jean-Michel Morel.
\newblock Connected components of sets of finite perimeter and applications to
  image processing.
\newblock {\em Journal of the European Mathematical Society}, 3:39--92, 2001.

\bibitem{Ambrosio2000}
Luigi Ambrosio, Nicola Fusco, and Diego Pallara.
\newblock {\em Functions of {B}ounded {V}ariation and {F}ree {D}iscontinuity
  {P}roblems}.
\newblock Oxford Mathematical Monographs, 2000.

\bibitem{Ambrosio2013}
Luigi Ambrosio and Francesco Ghiraldin.
\newblock Compactness of {S}pecial {F}unctions of {B}ounded {H}igher
  {V}ariation.
\newblock {\em Analysis and Geometry in Metric Spaces}, pages 1--30, 2013.

\bibitem{Aubert2000}
Gilles Aubert and Pierre Kornprobst.
\newblock {\em Mathematical Problems in Image Processing: Partial Differential
  Equations and the Calculus of Variations}, volume 147 of {\em Applied
  Mathematical Sciences}.
\newblock Springer-Verlag, New York, Inc., New York, 2002.

\bibitem{Bresson2006}
Xavier Bresson, Pierre Vandergheynst, and Jean-Philippe Thiran.
\newblock A {V}ariational {M}odel for {O}bject {S}egmentation {U}sing
  {B}oundary {I}nformation and {S}hape {P}rior {D}riven by the {M}umford-{S}hah
  {F}unctional.
\newblock {\em International Journal of Computer Vision}, 68(2):145--162, 2006.

\bibitem{Brown2012}
Ethan~S. Brown, Tony~F. Chan, and Xavier Bresson.
\newblock Completely {C}onvex {F}ormulation of the {C}han-{V}ese {I}mage
  {S}egmentation {M}odel.
\newblock {\em International Journal of Computer Vision}, 98:103--121, 2012.

\bibitem{Chan2001}
Tony~F. Chan and Luminata~A. Vese.
\newblock {A}ctive {C}ontours {W}ithout {E}dges.
\newblock {\em IEEE Transactions on Image Processing}, 10(2):266--277, February
  2001.

\bibitem{Chen2002}
Yunmei Chen, Hemant~D. Tagare, Sheshadri Thiruvenkadam, Feng Huang, David
  Wilson, Kaundinya~S. Gopinath, Richard~W. Briggs, and Edward~A. Geiser.
\newblock Using {P}rior {S}hapes in {G}eometric {A}ctive {C}ontours in a
  {V}ariational {F}ramework.
\newblock {\em International Journal of Computer Vision}, 50(3):315--328, 2002.

\bibitem{Chung2009}
G.~Chung and L.A. Vese.
\newblock Image segmentation using a multilayer level-set approach.
\newblock {\em Computing and Visualization in Science}, 12:267--285, August
  2009.

\bibitem{Diop2013}
El~Hadji~S. Diop and Val\'erie Burdin.
\newblock Bi-planar image segmentation based on variational geometrical active
  contours with shape priors.
\newblock {\em Medical Image Analysis}, 17:165--181, 2013.

\bibitem{Evans1998}
Lawrence~C. Evans.
\newblock {\em Partial {D}ifferential {E}quations}, volume~19 of {\em Graduate
  Studies in Mathematics}.
\newblock American Mathematical Society, Providence, Rhode Island, June 1998.

\bibitem{Evans1992}
Lawrence~C. Evans and Ronald~F. Gariepy.
\newblock {\em Measure {T}heory and {F}ine {P}roperties of {F}unctions}.
\newblock Studies in advanced mathematics. CRC Press, Boca Raton, FL, 1992.

\bibitem{Giusti1994}
Enrico Giusti.
\newblock {\em Minimal {S}urfaces and {F}unctions of {B}ounded {V}ariation}.
\newblock Birkhuser, 1994.

\bibitem{Giusti2000}
Enrico Giusti.
\newblock {\em Direct {M}ethods in the {C}alculus of {V}ariations}.
\newblock World Scientific, 2005.

\bibitem{Leventon2000}
Michael~E. Leventon, W.~Eric~L. Grimson, and Olivier Faugeras.
\newblock Statistical {S}hape {I}nfluence in {G}eodesic {A}ctive {C}ontours.
\newblock In {\em CVPR}, pages 316--323, South Carolina, USA, June 2000.

\bibitem{Li2010}
Fang Li, Michael~K. Ng, Tie~Yong Zeng, and Chunli Shen.
\newblock A {M}ultiphase {I}mage {S}egmentation {M}ethod {B}ased on {F}uzzy
  {R}egion {C}ompetition.
\newblock {\em SIAM Journal on Imaging Sciences}, 3:277--299, 2010.

\bibitem{Moren2014}
Juan~C. Moreno, V.~B.~Surya Prasath, Hugo Proen{\c c}a, and K.~Palaniappan.
\newblock Fast and globally convex multiphase active contours for brain {MRI}
  segmentation.
\newblock {\em Computer Vision and Image Understanding}, 125:237--250, August
  2014.

\bibitem{Mumford1989}
David Mumford and Jayant Shah.
\newblock Optimal {A}pproximations by {P}iecewise {S}mooth {F}unctions and
  {A}ssociated {V}ariational {P}roblems.
\newblock {\em Communications on Pure and Applied Mathematics}, 42:577--685,
  1989.

\bibitem{Ozere2015}
Sol\`ene Ozer\'e, Christian Gout, and Carole~Le Guyader.
\newblock Joint {S}egmentation/{R}egistration {M}odel by {S}hape {A}lignment
  via {W}eighted {T}otal {V}ariation {M}inimization and {N}onlinear
  {E}lasticity.
\newblock {\em SIAM Journal on Imaging Sciences}, 8(3):1981--2020, 2015.

\bibitem{pelapur20173d}
Rengarajan Pelapur, VB~Surya Prasath, Juan~C Moreno, and Michael~M Heck.
\newblock 3d workflow for segmentation and interactive visualization in brain
  mr images using multiphase active contours.
\newblock In {\em IEEE International Conference on Bioinformatics and
  Biomedicine (BIBM)}, pages 921--926. IEEE, 2017.

\bibitem{Pock2009}
Thomas Pock, Antonin Chambolle, Daniel Cremers, and Horst Bischof.
\newblock A convex relaxation approach for computing minimal partitions.
\newblock In {\em 2009 IEEE Conference on Computer Vision and Pattern
  Recognition}, pages 810--817. IEEE, 2009.

\bibitem{prasath2015cell}
VB~Surya Prasath, Kiichi Fukuma, Bruce~J Aronow, and Hiroharu Kawanaka.
\newblock Cell nuclei segmentation in glioma histopathology images with color
  decomposition based active contours.
\newblock In {\em IEEE International Conference on Bioinformatics and
  Biomedicine (BIBM)}, pages 1734--1736. IEEE, 2015.

\bibitem{Sussman1994}
Mark Sussman, Peter Smereka, and Stanley Osher.
\newblock A {L}evel {S}et {A}pproach for {C}omputing {S}olutions to
  {I}ncompressible {T}wo-{P}hase {F}low.
\newblock {\em Journal of Computational Physics}, 114(1994):146--159, 1994.

\bibitem{Wu2015}
Yongfei Wu and Chuanjiang He.
\newblock A convex variational level set model for image segmentation.
\newblock {\em Signal Processing}, 106:123--133, 2015.

\bibitem{Zhao1996}
Hong-Kai Zhao, Tony Chan, Barry Merriman, and Stanley Osher.
\newblock A variational level set approach to multiphase motion.
\newblock {\em Journal of Computational Physics}, 127(1):179--195, 1996.

\end{thebibliography}
\end{document}